\documentclass[12pt]{amsart}
\usepackage{amsmath, amssymb,latexsym,amsthm}

\usepackage{caption}

\usepackage[dvipdfmx]{graphicx}

\usepackage[dvipdfmx]{overpic}

\newcommand{\real}{\mathbb{R}}

\newcommand{\cP}{\mathcal{P}}
\newtheorem{Theorem}{Theorem}

\newtheorem{Lemma}[Theorem]{Lemma}

\theoremstyle{remark}
\newtheorem{Remark}[Theorem]{Remark}
\begin{document}

\title[partially hyperbolic endomorphisms]{On bifurcation of statistical properties of partially hyperbolic endomorphisms}
\author{Masato Tsujii}
\thanks{MT is supported by JSPS KAKENHI Grant Number 21H00994.}
\address{Department of Mathematics, Kyushu university, 744 Motooka, Nishi-ku, Fukuoka, 819-0395, JAPAN}
\author{Zhiyuan Zhang}
\address{CNRS, Institut Galil\'ee Universitat Paris 13,  99, Avenue Jean-Baptiste cl\'ement 93430 - Villetaneuse, FRANCE}

\begin{abstract}
We give an example of a path-wise connected  open set of $C^\infty$ partially hyperbolic endomorphisms on the $2$-torus, on which the SRB measure exists for each system and varies smoothly depending on the system, while the sign of its central Lyapunov exponent does change.  
\end{abstract}

\maketitle

\section{Introduction}
We give an example of a path-wise connected  open set of $C^\infty$ partially hyperbolic endomorphisms\footnote{We refer \cite{MR2231338} for a few technical terms that are used in this paper such as partially hyperbolic endomorphism, Lyapunov exponents, etc.} on the $2$-torus, on which the SRB measure exists for each system and varies smoothly depending on the system, while the sign of its central Lyapunov exponent does change.  Since the Lyapunov exponents of the SRB measure are major characteristics that describe local geometric structure of dynamics almost everywhere, we tend to think that the dynamical systems exhibit drastic bifurcations when the Lyapunov exponents of the SRB measure change their signs. However our example tells that this is not always the case and suggests that  bifurcations of global statistical properties of partially hyperbolic dynamical systems  may be much milder than  what we expect  from our knowledge on low dimensional dynamical systems with one dimension of unstability such as H\'enon maps.

The open subset of $C^\infty$ partially hyperbolic endomorphisms in our example is a small open neighborhood of a one-parameter family of skew products of circle endomorphisms over an angle-multiplying map. 
Smooth dependence of the SRB measure  in a similar setting is already studied in \cite{MR3772032} partly based on the argument in \cite{AGT,GouezelLiverani06}. In this paper, we construct an example in which we can observe the switching of the sign for the central Lyapunov exponent of the SRB measure. In the last section, we present some results of numerical computations that illustrate the situation in our example.

\section{Result}
We write $\mathbb{T}=\real/\mathbb{Z}$ for the unit circle and $\mathbb{T}^2$ for 2-dimensional torus. 
We consider the iteration of a $C^\infty$ locally diffeomorphic  map $F:\mathbb{T}^2\to \mathbb{T}^2$ as a discrete dynamical system. The Perron-Frobenius operator 
\begin{equation}\label{eq:PF}
\cP:C^r(\mathbb{T}^2)\to C^r(\mathbb{T}^2), \quad 
\cP u(p)= \sum_{p'\in \mathbb{T}^2: F(p')=(p)}
\frac{u(p')}{|\det DF(p')|} 
\end{equation}
expresses  the action of $F$ on the space of densities, where $C^r(\mathbb{T}^2)$ denotes the space of $C^r$ functions on  $\mathbb{T}^2$.

An invariant Borel probability measure $\mu$ is said to be an  SRB measure if almost every point on $\mathbb{T}^2$ with respect to the Lebesgue measure  is generic for $\mu$. 
We consider a partially hyperbolic endomorphism $F$ on $\mathbb{T}^2$ and suppose that $F$ admits an ergodic SRB measure $\mu_F$. Then the Lyapunov exponents take constant values  
\[
\chi^c(\mu_F)\lneq \chi^u(\mu_F)\quad\text{with}\quad 
\chi^u(\mu_F)>0
\]
at almost every point with respect to $\mu_F$ and also with respect to the Lebesgue measure. 

Our main result is stated as follows.
\begin{Theorem}\label{th:main}
For any $r>0$, there exists a path-wise connected $C^\infty$ open subset $\mathcal{U}$ of $C^\infty(\mathbb{T}^2,\mathbb{T}^2)$ that consists of locally diffeomorphic partially hyperbolic endomorphisms,   a Hilbert space
\begin{equation}\label{eq:H}
C^\infty(\mathbb{T}^2)\subset \mathcal{H}\subset C^r(\mathbb{T}^2)
\end{equation}
and a constant $0<\rho<1$  such that 
\begin{itemize}
\item[(a)] The Perron-Frobenius operator $\mathcal{P}_F$ for $F\in \mathcal{U}$ restricts to a bounded operator
\begin{equation}\label{cpf}
\cP_F:\mathcal{H}\to \mathcal{H}.
\end{equation}
\item[(b)]
The restriction \eqref{cpf} has a simple eigenvalue $1$ and the rest of its spectral set is contained in the disk $|z|<\rho<1$.
\item[(c)] $F\in \mathcal{U}$ admits a unique SRB measure $\mu_F=\rho_F \mathrm{Leb}$ where $\rho_F\in \mathcal{H}$ is the eigenfunction of $\cP_F$ for the simple eigenvalue~$1$.
\item[(d)] The SRB measure $\mu_F$ depends on $F\in \mathcal{U}$ smoothly in the sense that, for any $C^\infty$ one-parameter family $G_t$ of maps in $\mathcal{U}$ and $\psi\in C^\infty(\mathbb{T}^2)$, the correspodence $t\mapsto \int \psi d\mu_{G_t}$ is $C^r$.  
\item[(e)]  There are $F_{\sigma}\in \mathcal{U}$ for $\sigma\in \{+,-\}$ such that the central Lyapunov exponent
$\chi^c(\mu_{F_\sigma})$ has the same sign as $\sigma$.  \end{itemize}
\end{Theorem} 
The claims of the theorem above imply that, if we take any $C^\infty$ one-parameter family $G_t$ that connects $F_-$ and $F_+$ in $\mathcal{U}$, we observe that the SRB measure $\mu_{G_t}$ changes smoothly with respect to $t$ while  the central Lyapunov exponent will change its sign at some parameter.

\section{Circle endomorphisms}\label{sec:circ}
We first consider the doubling map on the circle $\mathbb{T}$:  
\[
f_0:\mathbb{T}\to \mathbb{T}, \quad f_0(y)=2y \quad \mod\mathbb{Z}. 
\]
Below we deform the map $f_0$ in order to make a neutral fixed point in a small neighborhood of $0\in \mathbb{T}$.

Let $\varphi:\real\to \real$ be a $C^\infty$ map with the following properties:
\begin{enumerate}
\renewcommand{\labelenumi}{(\roman{enumi})}
\item $0\le \varphi(y)\le 1$ and $|\varphi'(y)|\le 4/3$ for $y\in \real$,
\item $\varphi(y)=0$ for $y\notin [1/10,1]$, 
\item  $\varphi(1/2)=1/2$, $\varphi'(1/2)=1$, $\varphi''(1/2)<0$ and 
\item $\varphi(y)<y$ for $y\in (0,1)\setminus \{1/2\}$.

\end{enumerate}
For a small real number $\varepsilon>0$, we define 
\[
f_{\varepsilon}:\mathbb{T}\to \mathbb{T}, \quad f_\varepsilon(y)=
\begin{cases}
f_0(y)-\varepsilon \cdot \varphi(\varepsilon^{-1}y),&\quad \text{if $y\in [0,\varepsilon]$};\\
f_0(y),&\quad\text{otherwise.}
\end{cases}
\]
For the dynamics of $f_\varepsilon$, we observe that there are only two fixed points $0$ and $P=\varepsilon/2$: $0$ is a hyperbolic repelling fixed point and $P=\varepsilon/2$ is a one-sided attracting neutral fixed point  with immediate basin $(0,P]$. 

We henceforth suppose that the parameter $\varepsilon>0$ is sufficiently small, say $0<\varepsilon<1/100$. 
Then, for $a\in \real$, we set
\[
f_{\varepsilon,a}:\mathbb{T}\to \mathbb{T}, \quad f_{\varepsilon,a}(y)=f_{\varepsilon}(y)+a\varepsilon.
\]
From the assumption (iv),  we have
\[
\frac{2}{3}\le f'_{\varepsilon,a}(y)\le \frac{10}{3}\quad \text{for any }y\in \mathbb{T}. 
\]
Hence, if $a\ge 1$, we have that $f_{\varepsilon,a}^{-1}([0,\varepsilon])\cap (0,\varepsilon)=\emptyset$ and hence \begin{equation}\label{eq:expanding}
(f_{\varepsilon,a}^2)'(y)\ge 2\cdot \frac{2}{3}=\frac{4}{3}>1\quad \text{for any }y\in \mathbb{T}.
\end{equation}
The family $a\mapsto f_{\varepsilon,a}$ exhibits the saddle-node bifurcation of the fixed point $0$ at the parameter $a=0$.  It is not difficult to check that $f_{\varepsilon,a}$ is uniformly expanding  if $0<a\le 2$. 
If $a<0$ and $|a|$ is sufficiently small, then $f_{\varepsilon,a}$ admits three fixed points
\[
P_0=-a\varepsilon <P_-<P_+
\]
in a small neighborhood of $0$, where $P_0$ and $P_+$ are hyperbolic repelling while $P_-$ is hyperbolic attracting. The immediate basin of the hyperbolic attracting fixed point $P_-$ is the interval $B=(P_0,P_+)$ and we have
\[
\lim_{a\to -0} P_0=0,\quad \lim_{a\to -0} P_-=\lim_{a\to -0} P_+=\frac{\varepsilon}{2}. 
\]

\begin{figure}
\begin{overpic}[scale=.10]{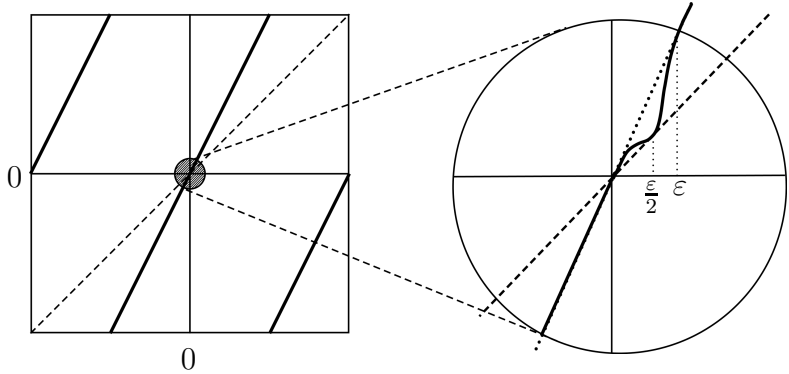}
\put(85,21){$\varepsilon$}
\put(81,20){$\frac{\varepsilon}{2}$}
\put(20,-2){$0$}
\put(-3,22){$0$}
\end{overpic}
\caption{The graph of the function $f_\varepsilon$.}
\label{fig0}
\end{figure}

\section{Skew products over angle multiplying maps}
We consider the dynamics of perturbations of the skew product 
\[
F_{\varepsilon,a,\delta,m}:\mathbb{T}^2\to \mathbb{T}^2,\quad 
F_{\varepsilon,a,\delta,m}(x,y)=(mx, f_{\varepsilon,a}(y)+ \delta \varepsilon \cos 2\pi x)
\] 
where $m$ is a positive integer and $\delta>0$ is a positive real parameter.
In the following, we suppose that $r>0$ is a given integer. We suppose that the constants $\varepsilon>0$ and $\delta>0$ are small, say $\varepsilon,\delta\in (0,1/100)$. We will also fix $m$ as a large constant so that the conclusion of Theorem \ref{th:1} below holds true.  Since  we  regard $F_{\varepsilon,a,\delta,m}$ as a one parameter family with parameter $a\in [-2\delta,2]$, we henceforth write $F_a$ for $F_{\varepsilon,a,\delta,m}$. 
 
 \subsection{Quasi-compactness of $\mathcal{P}$}
We adapt the argument in 
\cite{MR3772032} to get the next theorem. 
Since the situation is only a little different from that in \cite{MR3772032}, we give a brief account on its proof in Section \ref{sec:pf}. \begin{Theorem}\label{th:1}
If we let $m$ be sufficiently large depending on the parameters $r$, $\varepsilon$, $\delta$ and a given $0<\rho_0<1$, there exists a Hilbert space $\mathcal{H}$ satisfying \eqref{eq:H} and a $C^\infty$ neighborhood $\mathcal{U}\subset C^\infty(\mathbb{T}^2,\mathbb{T}^2)$ of the family $\mathcal{F}=\{F_a=F_{\varepsilon,a,\delta,m}, a\in[-2\delta,2]\}$, such that the Perron-Frobenius operator $\cP_F:\mathcal{H}\to \mathcal{H}$ for $F\in \mathcal{U}$ is bounded  and its essential spectral radius is bounded by $\rho_0$. 

Further, if $1$ is a simple eigenvalue of the Perron-Frobenius operator $\cP_F:\mathcal{H}\to \mathcal{H}$ for every $F\in \mathcal{F}$, then, by letting the neighborhood $\mathcal{U}$ be smaller, we may suppose that the same is true for all $F\in \mathcal{U}$ and the positive eigenfunction $\rho_F\in \mathcal{H}$ for the simple eigenvalue $1$, determined by the condition $\int \rho_F d\mathrm{Leb} =1$, depends on $F$ smoothly in the following sense: for any $C^\infty$ one-parameter family $G_t$ of maps in $\mathcal{U}$ and $\psi\in C^\infty(\mathbb{T}^2)$, the correspodence $t\mapsto \int \psi d\mu_{G_t}=\int \psi \rho_{G_t} d\mathrm{Leb}$ is $C^r$.  
\end{Theorem}
\begin{Remark}
We can not let $r=\infty$ in our construction  because it is essential to take $m$ large enough depending on $r$. 
\end{Remark}
\subsection{Simplicity of the eigenvalue $1$}
We show the following theorem for the family $\mathcal{F}=\{F_a\mid  a\in [-2\delta,2]\}$. 
\begin{Theorem}\label{th:2}
For any $a\in [-2\delta,2]$, the principal eigenvalue $1$ of $\cP_F:\mathcal{H}\to \mathcal{H}$ is simple and there is no other eigenvalue on the unit circle. 
The eigenfunction $\rho_a\in\mathcal{H}$ for the simple eigenvalue $1$ satisfying  $\int \rho_a d\mathrm{Leb}=1$ is the density of the SRB measure   $\mu_{a}$ with respect to the Lebesgue measure. 
\end{Theorem}
\begin{proof} 
We consider the following two cases for $a\in [-2\delta,2]$ separately:
\[
\textrm{(i)}\; a+\delta>0,\quad
\textrm{(ii)}\; a+\delta\le 0.
\]

\subsection*{Case (i)} First we prove 
\begin{Lemma}\label{lm:erg}
In Case (i), we have  $U_\infty:=\cup_{n\ge 0}F_a^n(U)=\mathbb{T}^2$ for any non-empty open subset $U$ on $\mathbb{T}^2$. 
\end{Lemma}
\begin{proof}
Since $F_a$ is expanding in the horizontal (or $x$-) direction, we have that $U_\infty\cap (\{0\}\times \mathbb{T})\neq \emptyset$. The map  $F_{a}$ restricted to $\{0\}\times \mathbb{T}$ can be identified with  $f_{\varepsilon,a+\delta}$.
From the assumption, we have $a+\delta>0$ and hence $f_{\varepsilon,a+\delta}$ is uniformly expanding, provided that $\delta>0$ is sufficiently small.
\begin{Remark}
The last claim is not completely obvious but easy to check. Let $f=f_{\varepsilon,a}$. 
To show that $f$ is uniformly expanding, 
it is enough to show that there exists $n>0$ for any $x\in \mathbb{T}$ such that $(f^n)'(x)>1$. This holds obviously with $n=1$ for $x$ on the outside of $(0,\varepsilon)$. For a point $x\in (0,\varepsilon)$, we let $m$ be the smallest integer such that $f^{m}(x)\notin (0,\varepsilon)$. By the elementary estimates on intermittent one dimensional map, we see that
$(f^{m})'(x)>h>0$ for some constant $h>0$ independent of $a>0$ and $\varepsilon>0$ (as far as they are sufficiently small). By letting $\varepsilon>0$ be sufficiently small, we may suppose that the orbit starting from $f^{m}(x)$ will not return to $(0,\varepsilon)$ for arbitrarily long time and therefore we can find $n>m$ such that $(f^{n})'(x)>1$. 
\end{Remark}

Hence we have $U_\infty\supset  \{0\}\times \mathbb{T}$. Again, using the fact that $F_a$ is expanding in the horizontal direction, we obtain the claim $U_\infty=\mathbb{T}^2$.
\end{proof} 

Suppose that $\rho\in \mathcal{H}$ is an eigenfunction for an eigenvalue on the unit circle. Then we have $|\mathcal{P}^n\rho|=|\rho|=\mathcal{P}^n|\rho|$ for $n\ge 1$. From the last lemma, this holds only if $\rho=e^{i\theta} |\rho|$ for some $\theta\in [0,2\pi)$ and therefore we may suppose $\rho\ge 0$. This implies that there is no eigenvalue on the unit circle other than $1$. By the same reason,  the geometric multiplicity of the eigenvalue $1$ should be $1$. Further, since $\mathcal{P}$ preserves the integral of functions with respect to the Lebesgue measure, we conclude that the algebraic multiplicity is not greater than $1$. 

Let $\rho_{F_a}\in \mathcal{H}\subset C^r(\mathbb{T}^2)$ be the eigenfunction of $\cP_{F_a}$ for the simple eigenvalue $1$. We may and do suppose that $\rho_{F_a}$ is non-negative and $\int \rho_{F_a} \,d\mathrm{Leb}=1$. Then the measure $\nu_{F_a}:=\rho_{F_a} \,\mathrm{Leb}$ is ergodic since $\mathcal{P}^n u$ converges to a constant multiple of $\rho_{F_a}$ for any $u\in\mathcal{H}$.
Since $\rho_{F_a}\in C^r(\mathbb{T})$, there is an open subset $U\subset \mathbb{T}^2$ on which $\rho_{F_a}>0$ and therefore almost every point in $U$ is generic for $\mu_{F_a}$. As $F_a$ is locally diffeomorphic, almost every point on $F_a^n(U)$ with $n\ge 0$ is generic for $\mu_{F_a}$. Since $\cup_{n\ge 0}F_a^n(U)=\mathbb{T}^2$ as we showed in Lemma \ref{lm:erg}, we conclude that almost every point on $\mathbb{T}^2$ is generic for $\mu_{F_a}$. This finishes the proof of the theorem in Case (i).

\subsection*{Case (ii)} Note that $a\le -\delta<0$ in this case. The region 
\[
W=\mathbb{T}\times ((\delta-a)\varepsilon, \varepsilon/2)
\]
satisfies $F_a(W)\subset W$ and the  iteration of $F_a$ is (non-uniformly) contracting on the fibers $\{x\}\times ((\delta-a)\varepsilon, \varepsilon/2)$ for $x\in \mathbb{T}$. 
\begin{Remark}
The choice of the interval $(\delta-a)\varepsilon,\varepsilon/2)$ in the definition of $W$ is made as follows: The left end point $y_-=(\delta-a)\varepsilon$ is the unique point in $(0,\varepsilon/10)$ satisfying $
f_{\varepsilon,a-\delta}(y_-)=y_-$.
(Recall the condition (ii) in the definition of the function $\varphi$.)
The right end point $y_+=\varepsilon/2$  is the neutral fixed point of $f_{\varepsilon,0}$, which satisfies $f_{\varepsilon,a+\delta}(y_+)\le y_+$ when $a+\delta\le 0$. 
\end{Remark}
Hence there exists a unique mixing $F_a$-invariant measure $\mu_{F_a}$ supported in $W$ such that Lebesgue almost every point on $W$ is generic for $\mu_{F_a}$. 

Writing $\pi_2:\real^2\to \real$ for the projection to the second component, we have $\pi_2\partial_y (F_{a}\circ F_{a})(p)>1$ on the complement of $W$, with only one exception $p=(0,\varepsilon/2)$ when $a+\delta=0$. Hence the intersection of the complement 
\[
C=\mathbb{T}^2\setminus \cup_{n\ge 0} F_{a}^{-n}(W)
\]
with any fiber $\{x\}\times \mathbb{T}$ can not contain any non-trivial interval. 

We next show that the complement $C$ is of null Lebesgue measure.  
Suppose that $C$ has positive Lebesgue measure and write $\mathbf{1}_{C}$ for the characteristic function of it. Then we can find a weak limit point $\rho$ of the sequence $(1/n)\sum_{k=0}^{n-1}\mathcal{P}^k\mathbf{1}_{C}$. By approximating $\mathbf{1}_{C}$ by $C^\infty$ function in $L^1$ sense  and using the spectral property of $\cP$ in Theorem \ref{th:1}, we see that $\rho$ belongs to $\mathcal{H}\subset C^r(\mathbb{T}^2)$ and is supported on $C$. But this is impossible because $C$ has no interior point. 
 
Since the complement $C$ is of null Lebesgue measure, almost every point on $\mathbb{T}^2$ is generic for the mixing measure $\mu_F$. Clearly this implies the conclusion of the theorem. 
\end{proof}

Finally we prove the following theorem on the central Lyapunov exponent of the SRB measure $\mu_{F_a}$ for $F_a$ with $a\in [-2\delta,2]$. Note that we always assume that $\varepsilon>0$ and $\delta>0$ are small.
\begin{Theorem}\label{th:3}
{\rm (a)} If $a+\delta<0$, the central Lyapunov exponent $\chi^c(\mu_{F_a})$ is negative. \\
{\rm (b)} If $a\ge 1$,  the central Lyapunov exponent $\chi^c(\mu_{F_a})$ is positive. 
\end{Theorem}
\begin{proof} (a) As we observed in the proof of Theorem \ref{th:2} in Case (ii), there is a unique SRB measure $\mu_{F_a}$ whose support is contained in $W$ and its central Lyapunov exponent is negative. \\
(b) By \eqref{eq:expanding}, the map $F_{a}$ is expanding along the fibers in this case and therefore the central Lyapunov exponent of the SRB measure is positive, provided that $\delta>0$ is sufficiently small. 
\end{proof}

\section{The proof of Theorem \ref{th:1}}\label{sec:pf}
We can obtain the proof of Theorem  \ref{th:1} by following the argument in \cite{MR3772032} with slight modifications. 
Below we explain briefly how we modify the argument in \cite{MR3772032}. 

First we check a transversality condition. 
We consider the constant cones in the tangent bundle 
\[
\mathbf{C}=\bigcup_{p\in \mathbb{T}^2} \mathbf{C}_p 
=\{(p,v)=((x,y),(v_x,v_y))\in T\mathbb{T}^2\mid |v_y|\le C_0\delta\varepsilon |v_x|\}
\] 
where we fix a large constant $C_0$ so that $DF(\mathbf{C})\subset \mathbf{C}$. 
For given $p\in \mathbb{T}^2$ and $q,q'\in \mathbb{T}^2$, we write $q\pitchfork q'$ if
\[
DF_q(\mathbf{C})\cap DF_{q'}(\mathbf{C})=\{0\}.
\]
We define
\[
\mathbf{m}(F)=\frac{1}{(2/3)\cdot m}\cdot \sup_{p\in \mathbb{T}^2}\sup_{q\in F^{-1}(p)} 
\#\{q'\in F^{-1}(p)\mid q'\pitchfork q\}
\]
where $(2/3)\cdot m$ stands for a lower bound of $\det dF$. 
We can check the following lemma by simple computations. 
\begin{Lemma}
The quantity $\mathbf{m}(F)$ converges to $0$ when we let $m$ go to infinity and the convergence is  uniform for sufficiently small $\varepsilon>0$, $\delta>0$  and any  $a\in [-2,2]$. 
\end{Lemma}

We then follow the argument in \cite{MR3772032} almost literally, noting that $\mathbf{m}(F)$ corresponds to $m(f,1)$ defined in \cite[Sec.3]{MR3772032} and that we just consider the first iteration (or  the case $n=1$ there). 
Another difference of our setting from that in \cite{MR3772032} is the point that we consider non-linear endomorphisms on the fibers while they were rigid rotations in \cite{MR3772032}. But, since we just consider the first iteration, if we take sufficiently fine local charts and partition of unity in the argument in \cite[Sec.4]{MR3772032}, it is direct to get a parallel argument in our setting. 
Then the claim corresponding to \cite[Prop.3]{MR3772032} and Hennion's theorem gives the former claim of Theorem~\ref{th:1}. 
We can deduce the latter claim using the abstract perturbation theorem in \cite[Sec.8]{GouezelLiverani06} about perturbation of transfer operators. For this we again follow the argument in \cite[Sec.4.4]{MR3772032}. 

\section{Some numerical experiments}
We present some results of numerical experiments related to the claim of the main theorem.  
For simplicity of computation, we consider a similar but slightly different setting from that in the previous sections. We  consider a $C^\infty$  map 
$f:\mathbb{T}\to \mathbb{T}$ defined by 
\[
f(y)=2y -  \frac{\sin(2\pi y)+\cos(2\pi x)-1}{2\pi}\quad\mod \mathbb{Z}. 
\]
It has a neutral fixed point at $0$ and its dynamics is of very similar nature to that of $f_\varepsilon$ in Section \ref{sec:circ}. (See also Figure \ref{fig0}.)

Then 
we consider a family of dynamical systems $F_a:\mathbb{T}^2\to \mathbb{T}^2$ defined by
\[
F_a(x,y)=
\left(
7x, f(y)+\delta \cdot \cos(2\pi x)+a
\right)\quad \text{for } a\in [-2\delta,2\delta].
\]
where we set $\delta=10^{-2}$. In Figure \ref{fig1}, we compute the approximate central Lyapunov exponent at a randomly chosen point by iterating $F$ for  $10^6$ times and plot it against the parameters $-0.02\le a\le 0.02$ (resp. $-0.004\le a\le 0.004$) with step $10^{-3}$ (resp. $10^{-4}$).  We observe that the (central) Lyapunov exponent varies smoothly and changes its sign at a parameter $-0.001<a_0<0$. 

We also plot an orbit of randomly chosen initial point at the parameters $a=-0.02, -0.006,-0.003, -0.002$. (We draw the orbit from time $10^3$ to time  $10^6$.)
At the parameter $a=-0.02$, we observe that the orbits are trapped by a horizontal zonal region. 
When the parameter $a$ crosses the value  $-\delta=-0.01$,  we expect that the orbits start to spread over the whole space $\mathbb{T}^2$ and, as the parameter $a$ gets large, the density of the orbits become more uniform. But when the value of $a$ is close to $-0.01$, it is difficult to detect this phenomenon because only very small portion of orbits go out of (the ruin of) the attracting region and return to it again soon.  (See the picture for the parameter $a=-0.006$ in Figure \ref{fig2}.)

\begin{figure}[t]
    \begin{tabular}{cc}
      \begin{minipage}[t]{0.45\hsize}
        \centering
        \includegraphics[keepaspectratio, scale=0.35]{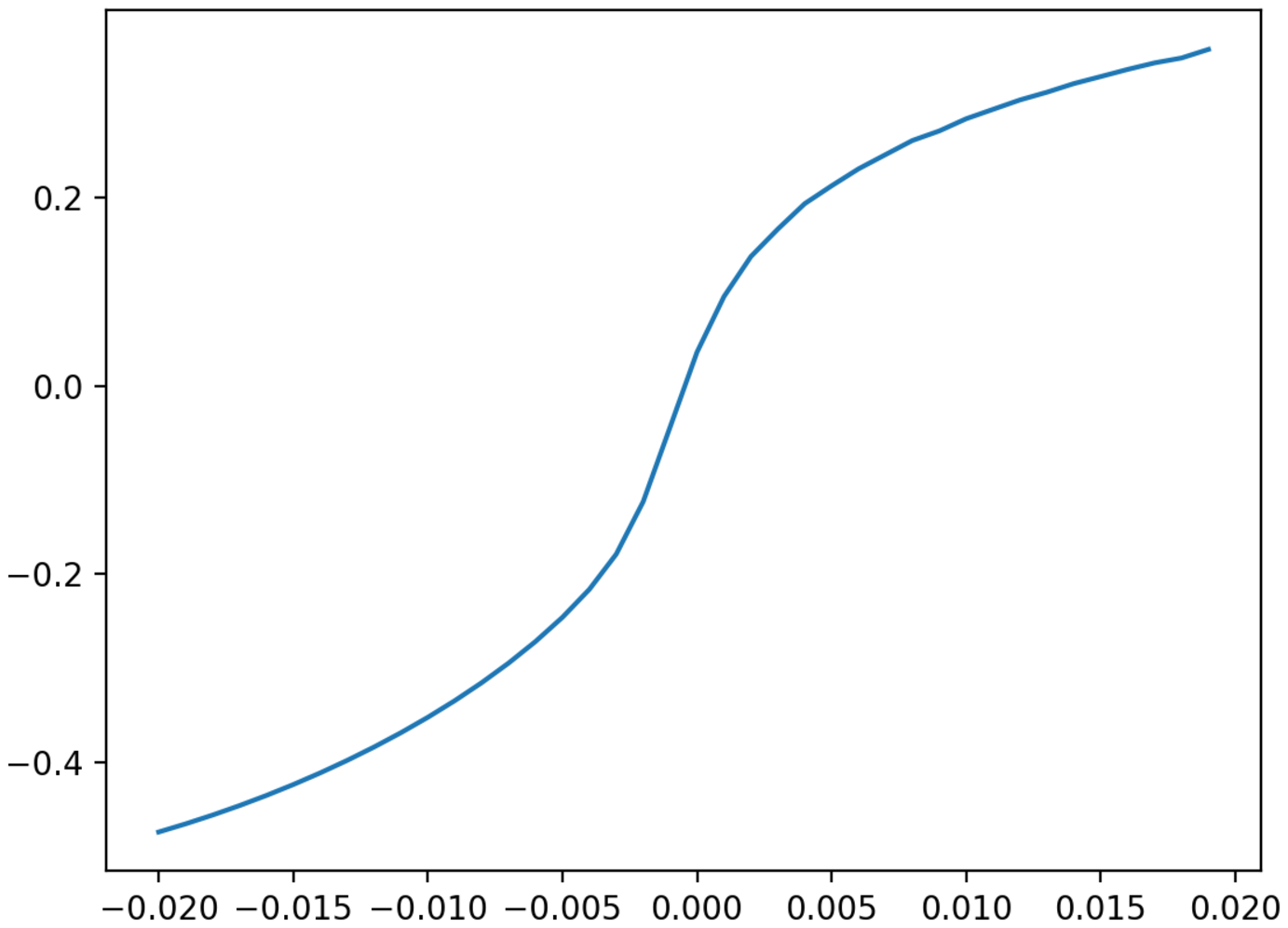}
      \end{minipage} &
      \begin{minipage}[t]{0.45\hsize}
        \centering
        \includegraphics[keepaspectratio, scale=0.35]{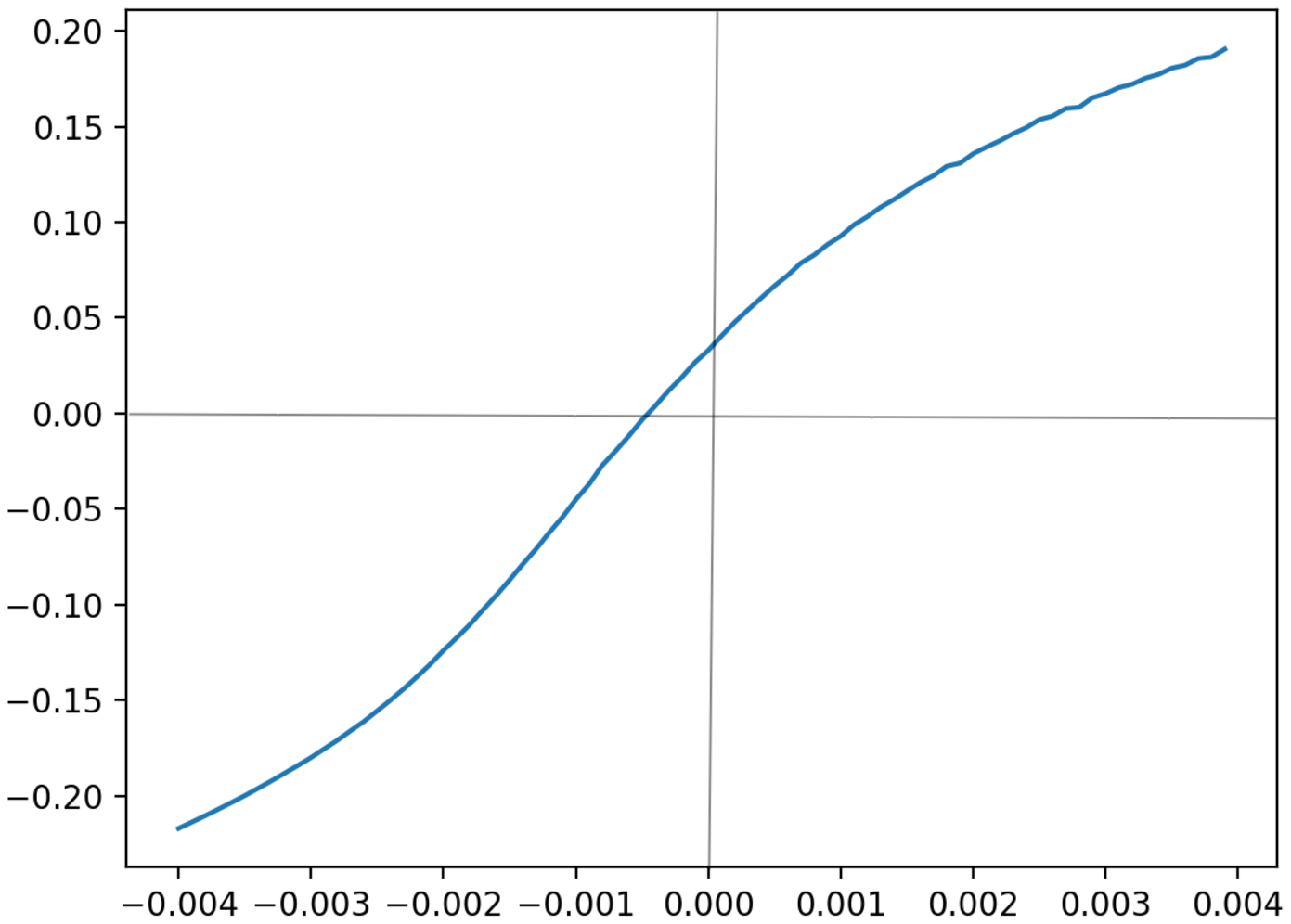}
      \end{minipage}
    \end{tabular}
   
    \caption{The central Lyapunov exponent of the SRB measure as a function of  the parameters $a$ close to $0$.  We draw the graph with the domain $[-0.02,0.02]$ and  $[-0.004,0.004]$ respectively  on the left and right pictures.}
    \label{fig1}
    \end{figure}
    
  \begin{figure}[t]
    \begin{tabular}{cc}
      \begin{minipage}[t]{0.45\hsize}
    \includegraphics[keepaspectratio, scale=0.35]
    {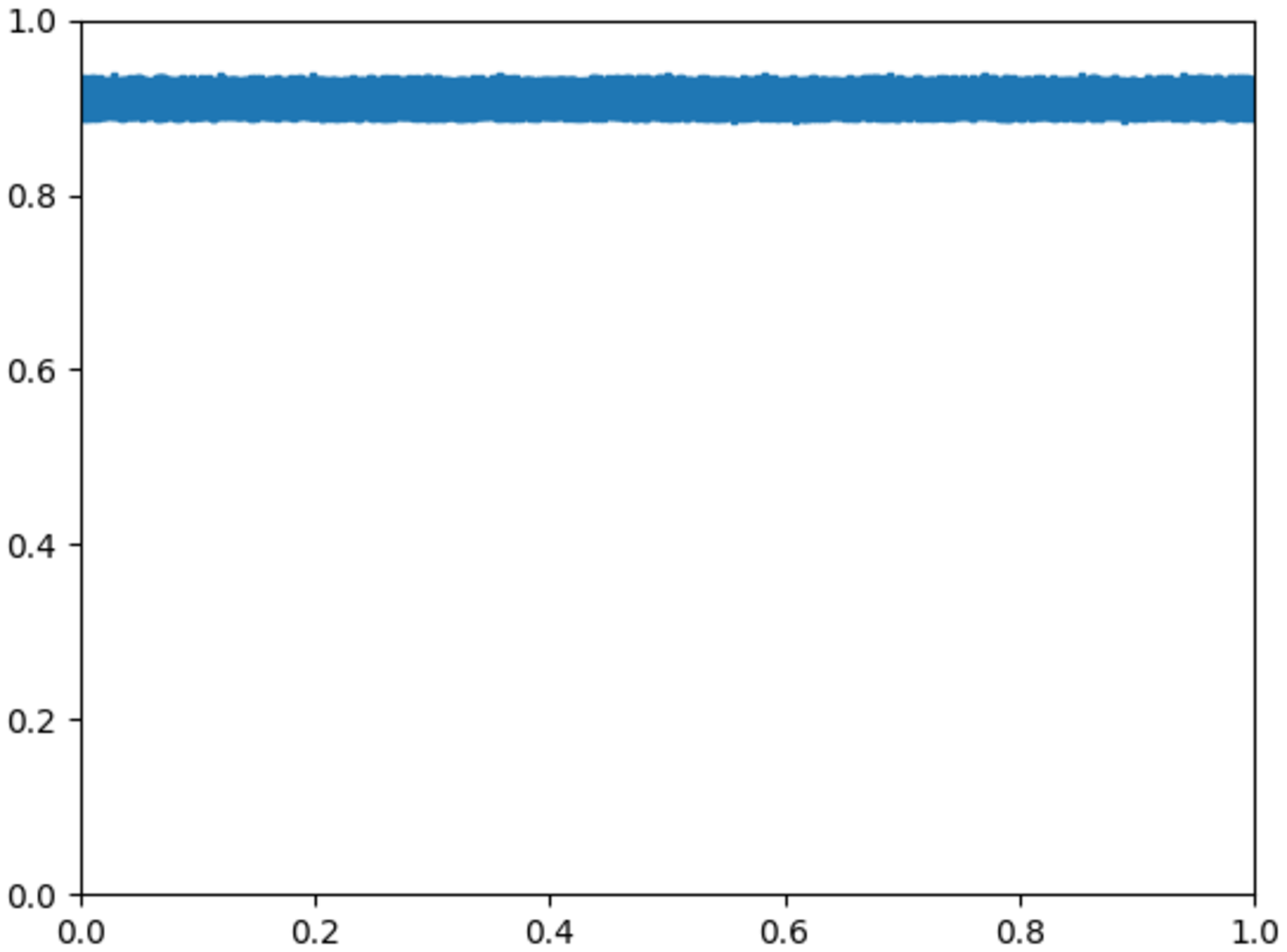}
        \caption*{$a=-0.02$}
      \end{minipage} &
      \begin{minipage}[t]{0.45\hsize}
        \centering
       \includegraphics[keepaspectratio, scale=0.35]{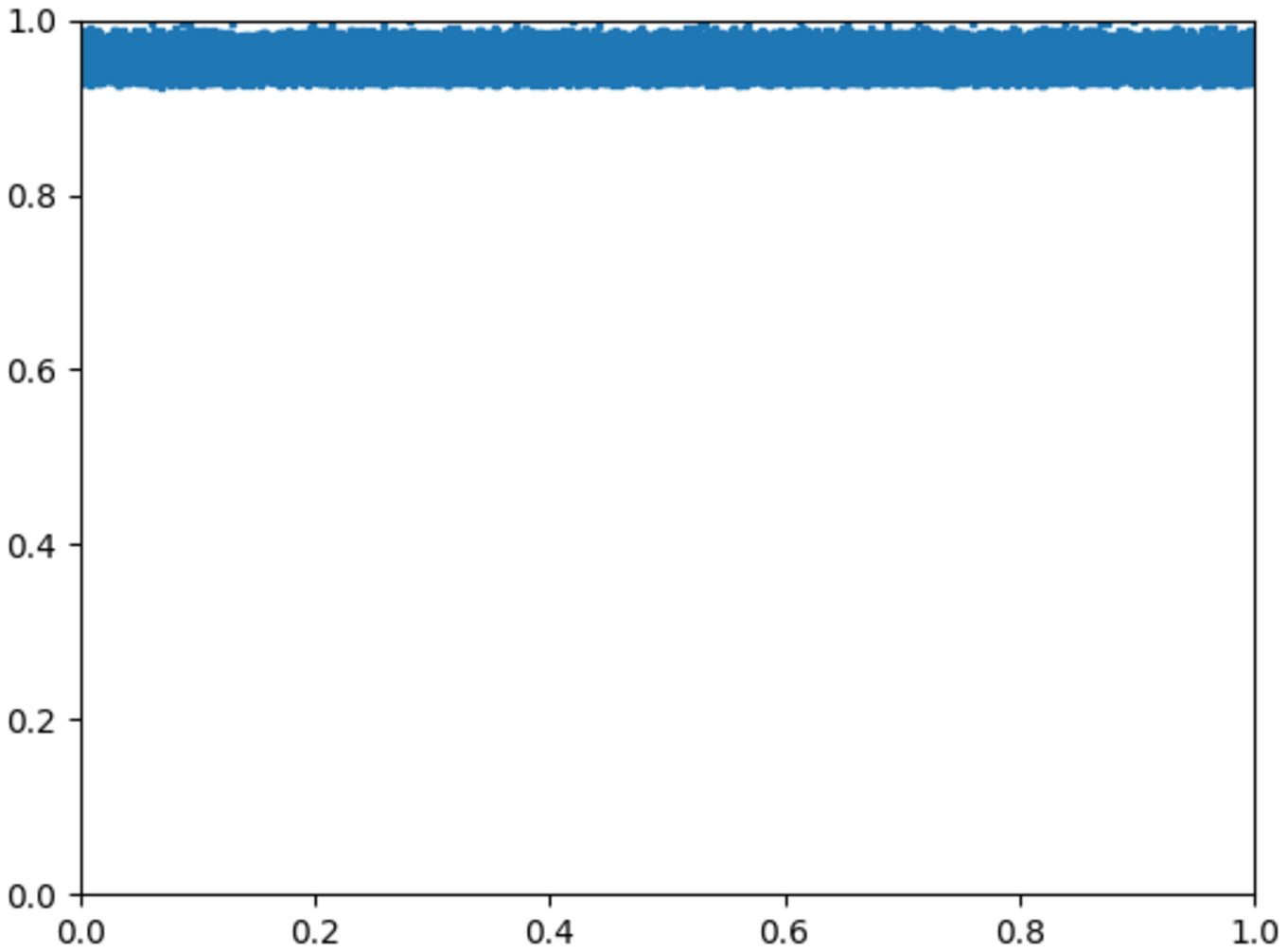}
        \caption*{$a=-0.006$}
      \end{minipage}\\
      \begin{minipage}[t]{0.45\hsize}
        \centering
        \includegraphics[keepaspectratio, scale=0.35]{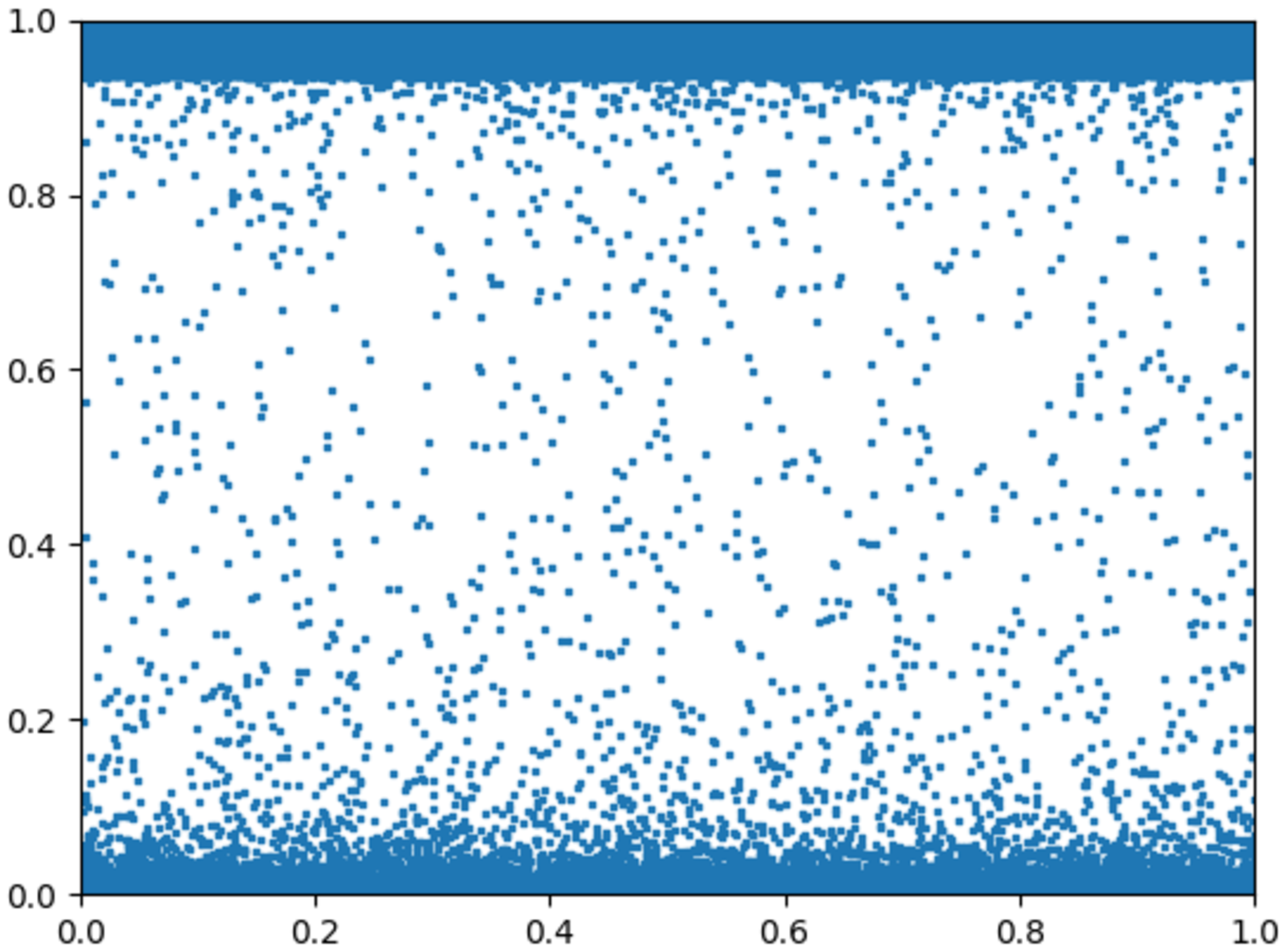}
        \caption*{$a=-0.003$}
      \end{minipage} &
      \begin{minipage}[t]{0.45\hsize}
        \centering
        \includegraphics[keepaspectratio, scale=0.35]{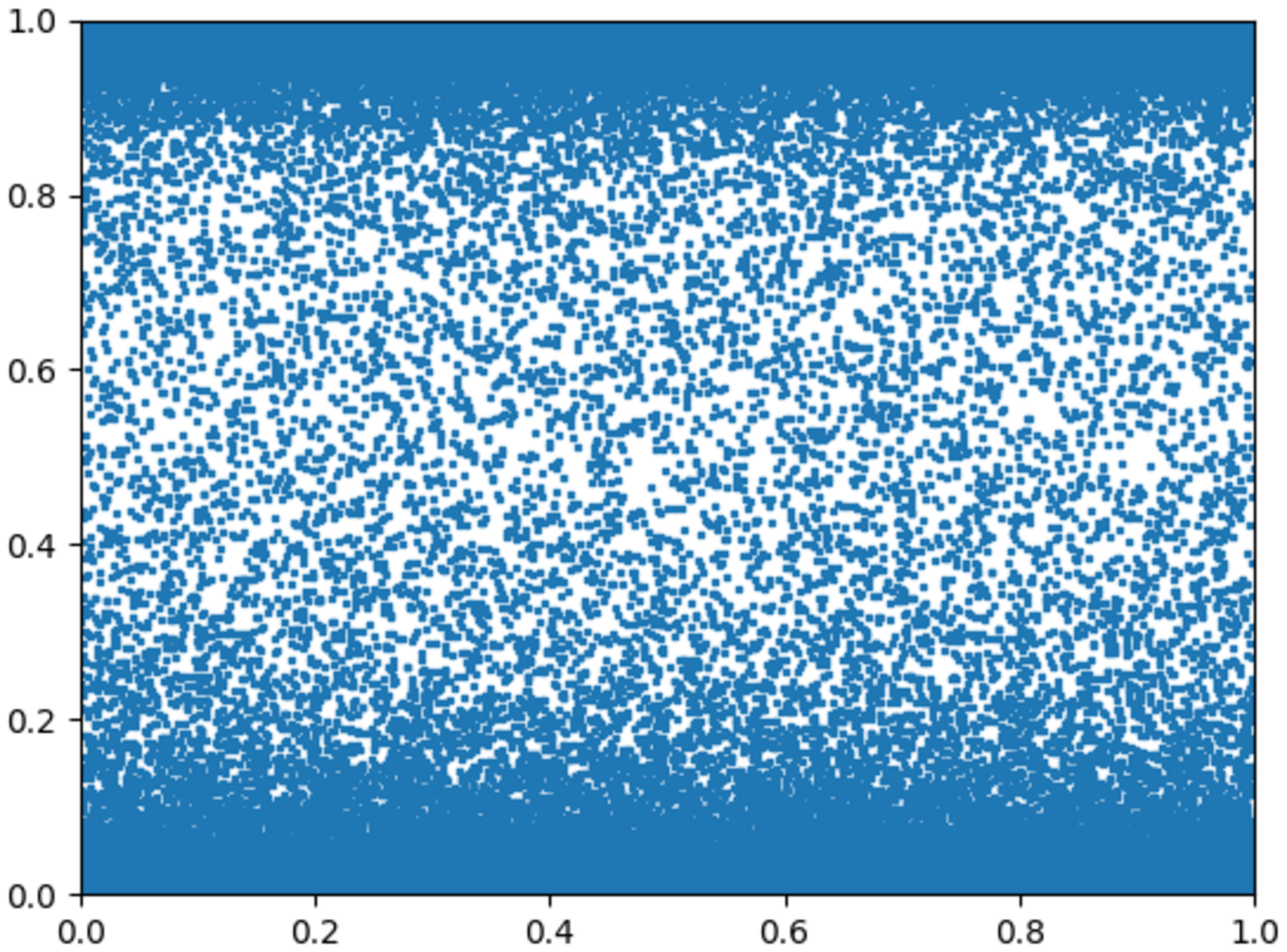}
        \caption*{$a=-0.002$}
      \end{minipage}

    \end{tabular}
    \caption{Plots of orbits of $F_a$ at a few values of the parameter $a$}
    \label{fig2}
  \end{figure}

\bibliography{mybib.bib}
\bibliographystyle{plain}
\end{document}